\documentclass[12pt]{article}
\usepackage{amsmath,amssymb}
\usepackage{color,graphicx}
\usepackage{ulem}

\newcommand{\bu}{\boldsymbol u}

\newtheorem{Theorem}{Theorem}

\newtheorem{lema}{Lemma}

\newcounter{remark}
\def\theremark {\arabic{remark}}
\newenvironment{remark}{\refstepcounter{remark}\par\noindent{\bf Remark\ \theremark}\ }{\par}
\newtheorem{Proof}{Proof}

\newenvironment{proof}{\begin{Proof}\rm}{\hfill $\Box$ \end{Proof}}

\usepackage{hyperref}
\usepackage[all]{hypcap}

\title{Optimal bounds for numerical approximations of infinite horizon problems based on dynamic programming approach}

\author{ Javier de Frutos\thanks{Instituto de Investigaci\'on en Matem\'aticas (IMUVA), Universidad
de Valladolid, Spain. Research supported
by Spanish MINECO
under grant PID2019-104141GB-I00 and {by Junta de Castilla y Le\'{o}n  under grant VA169P20 co-finanzed by FEDER (EU) funds} (frutos@mac.uva.es)}
  \and Julia Novo\thanks{Departamento de
Matem\'aticas, Universidad Aut\'onoma de Madrid, Spain.  Research supported
by Spanish MINECO
under grant PID2019-104141GB-I00 and {by Junta de Castilla y Le\'{o}n  under grant VA169P20 co-finanzed by FEDER (EU) funds} (julia.novo@uam.es)}}
\date{\today}
\begin{document}
\maketitle
\abstract{In this paper we get error bounds for fully discrete approximations of infinite horizon problems via the dynamic programming approach.
It is well known that considering a time discretization with a positive step size $h$ an error bound of size $h$ can be proved for the
difference between the value function (viscosity solution of the Hamilton-Jacobi-Bellman equation corresponding to the infinite horizon) and
the value function of the discrete time problem. However, including also a spatial discretization based on elements of size $k$ an error bound of size $O(k/h)$ can be found in the literature for the error between the value functions of the continuous problem and the fully discrete problem. In this paper we revise the error bound of the fully discrete method and prove, under similar assumptions to those of the time discrete case, that the error of the fully discrete case is in fact $O(h+k)$ which gives first order in time and space for the method. This error bound matches the numerical experiments of many papers in the literature in which the behaviour $1/h$ from the bound $O(k/h)$ have not been observed.}
\bigskip


\bigskip

\section{Introduction}
The numerical approximation of optimal control problems is  of importance for many applications such as aerospace engineering, chemical processing and resource economics, among others.
In this paper, we consider the dynamic programming approach to the solution of optimal control problems driven by dynamical systems in $\Bbb R^n$.  We refer to \cite{Bardi} for a monograph on this subject.

The value function of an optimal control problem is known to be usually only Lipschitz continuous even when the data is regular. The characterization of the value function is obtained in terms of a first-order nonlinear Hamilton-Jacobi-Bellman (HJB) partial differential equation.
A bottleneck in the computation of the value function comes from the need of approaching a nonlinear partial differential equation in dimension $n$, which is a challenging problem in high dimensions. Several approximation schemes have been proposed in the literature, ranging from finite differences to semi-Lagrangian and finite volume methods, see e.g. \cite{carlini}, \cite{Akian}, \cite{Guo}, \cite{Bo_et_al}. Some of these algorithms converge to the value function but their convergence is slow. The curse of dimensionality is mitigated in \cite{Alla_Falcone_Volkwein}, \cite{Alla_et_al} by means of a reduced-order model based on proper orthogonal decomposition. A new accelerated algorithm which can produce an accurate approximation of the value function in a reduced amount of time in comparison to other available methods is introduced in \cite{Dante}.

In the present paper, our concern is about the error bounds available in the literature for the fully discrete semi-lagrangian method approaching the value function, the viscosity solution of the HJB equation corresponding to the infinite horizon. For a method with a positive time step size $h$ and spatial elements of
size $k$ an error bound of size $O(k/h)$ can be found in \cite[Corollary 2.4]{falcone1}, \cite[Theorem 1.3]{falcone2}. However, the behaviour
$1/h$ in the error bound of the fully discrete method has never been observed in the numerical experiments, see for example \cite{Alla_Falcone_Volkwein}. Based on this fact, we reconsider the error analysis of the fully discrete method.

In this paper we prove a bound of size $O(h+k)$  which gives first order in time and space for the method.
{This rate of convergence is the same appearing in \cite[Corollary 2.4]{falcone1}. However, as stated
in \cite{Tidball}, \cite{falcone_cor}, the proof of this corollary was based on an identification which
does not hold in the example shown in \cite{Tidball}}. The idea of the present paper is to imitate the analysis of the discrete-time  method
for which the value function is characterized as the minimum of a functional, see \cite[Proposition 4.1, Chapter VI (appendix)]{Bardi}. To this end, we define a fully discrete cost functional that differs from that of the discrete-time approximation in the use of spatial interpolator operators. Then, in Theorem \ref{th_prin}
we prove that the fully discrete approximation can also be characterized as the minimum of this fully discrete cost functional. In the proof
we use that the fully discrete approximation defined by means of a discrete dynamic programming principle is unique, as it is proved in
\cite[Theorem 1.1, Appendix A]{Bardi}.

Finally, thanks to this new characterization we can extend the ideas from \cite[Lemma 1.2, Chapter VI]{Bardi} (for the
semi-discrete case) to the fully discrete case and
we are able to prove first order of convergence for the method both is space and time.
{First order of convergence in time is linked to the assumption of Lipschitz continuity of the controls in
the intervals defined excluding a finite number of jump discontinuities. In case we have less
regularity we obtain weaker error bounds. For example, for uniformly continuous controls, allowing again a finite
number of discontinuities, the error goes to zero as $h$ goes to zero but the first order of convergence is not achieved. Intermediate
rates of convergence or order $\alpha$ in time, for $0<\alpha<1$ are equally proved for $\alpha$-H\"{o}lder continuous controls.}

{Following the arguments in \cite{boba_etal} we can also prove convergence arguing with piecewise constants controls and under weaker
regularity assumptions (only some convexity assumptions are needed but no extra regularity assumptions for the controls). However, adapting the arguments in \cite{boba_etal} written for finite horizon problems to
our infinite horizon case we loose the full first order in time. We  develop this argument at the end of the paper.}

{We think that the new characterization we introduce in this paper, based on optimality arguments, could have potential to be used in other types of Hamilton-Jacobi-Bellman equations where the convergence rates are still suboptimal.}

{To conclude this section we want to mention some other related references. The paper \cite{gonzalez}  contains a first suboptimal convergence
of rate $\log(h)h^{1/2}$ for a similar scheme. The monograph of Falcone and Ferreti \cite{falcone_libro} contains a few of the difference convergence rates for HJB partial differential equations in control and games. 
}

The outline of the paper is as follows. In section 2 we state the model problem and some preliminary results. In Section 3 {we prove several error bounds for the method under different regularity requirements. More precisely, in  Subsection 3.1 we assume some regularity assumptions for the controls and show how the first order in time and space can be achieved. In Subsection 3.2, we follow the arguments
in \cite{boba_etal}  to prove convergence under weaker regularity assumptions.
}
\section{Model problem and Preliminary results}
Along this section we follow the notation in \cite{Alla_Falcone_Volkwein}. For a nonlinear mapping
$$
f:{\Bbb R}^n\times {\Bbb R}^m\rightarrow {\Bbb R}^n,
$$
and a given initial condition $y_0 \in {\Bbb R}^n$ let us consider the controlled nonlinear dynamical system
\begin{equation}\label{din_sis}
\dot y(t)=f(y(t),u(t))\in {\Bbb R}^n,\quad t>0,\quad y(0)=y_0\in {\Bbb R}^n,
\end{equation}
together with the infinite horizon cost functional
\begin{equation}\label{fun_cos}
J(y,u)=\int_0^\infty g(y(t),u(t))e^{-\lambda t}~dt.
\end{equation}
In \eqref{fun_cos} $\lambda>0$ is a given weighting parameter and
$$
g:{\Bbb R}^n\times {\Bbb R}^m\rightarrow {\Bbb R}.
$$
The set of admissible controls is
$$
{\Bbb U}_{\rm ad}=\left\{u\in {\Bbb U}\mid u(t)\in U_{\rm ad} \ {\rm for}\ {\rm almost}\ {\rm all}\ t\ge 0\right\},
$$
where ${\Bbb U}=L^2(0,\infty;{\Bbb R}^m)$ and $U_{\rm ad}\subset {\Bbb R}^m$ is a compact convex subset.

As in \cite[Assumption 2.1]{Alla_Falcone_Volkwein} we assume the following hypotheses:
\begin{itemize}
\item The right-hand side $f$ in \eqref{din_sis} is continuous and globally Lipschitz-continuous in both the first
and second arguments; i.e., there exists
a constant $L_f>0$ satisfying
\begin{eqnarray}\label{lip_f}
\|f(y,u)-f(\tilde y,u)\|_2&\le& L_f\|y-\tilde y\|_2,\quad \forall y,\tilde y\in{\Bbb R}^n, u\in U_{\rm ad}\\
\|f(y,u)-f(y,\tilde u)\|_2&\le& L_f \|u-\tilde u\|_2,\quad \forall u,\tilde u\in U_{\rm ad}, y\in {\Bbb R}^n.
\label{lip_f2}
\end{eqnarray}
\item The right-hand side $f$ in \eqref{din_sis}  satisfies that there exists a constant $M_f>0$ such that the following bound holds
\begin{equation}\label{infty_f}
\|f(y,u)\|_\infty=\max_{1\le i\le n}|f_i(y,u)|\le M_f, \quad \forall y\in \overline \Omega\subset {\Bbb R}^n, u\in U_{\rm ad},
\end{equation}
where $\overline \Omega$ is a bounded polyhedron such that for sufficiently small $h>0$ {the following inward pointing condition on the dynamics} holds
\begin{equation}\label{invariance}
y+hf(y,u)\in \overline \Omega,\quad \forall y \in \overline \Omega, u\in U_{\rm ad}.
\end{equation}
\item The running cost $g$ is continuous and globally Lipschitz-continuous in both the first and second arguments; i.e., there exists a constant $L_g>0$
satisfying
\begin{eqnarray}\label{lip_g}
|g(y,u)-g(\tilde y,u)|&\le& L_g\|y-\tilde y\|_2,\quad \forall y,\tilde y\in{\Bbb R}^n, u\in U_{\rm ad}\\
|g(y,u)-g(y,\tilde u)|&\le& L_g \|u-\tilde u\|_2,\quad \forall u,\tilde u\in U_{\rm ad}, y \in {\Bbb R}^n.\label{lip_g2}
\end{eqnarray}
\item Moreover, there exists a constant $M_g>0$ such that
\begin{equation}\label{cota_g}
|g(y,u)|\le M_g,\quad \forall (y,u)\in \overline \Omega\times U_{\rm ad}.
\end{equation}
\end{itemize}
From the assumptions made on $f$ there exists a unique solution of \eqref{din_sis} $y=y(y_0,u)$ {defined on $[0,\infty)$} for every admissible control
$u\in {\Bbb U}_{\rm ad}$ and for every initial condition $y_0\in {\Bbb R}^n$, see \cite[Chapter 3]{Bardi}. We define the reduced cost functional as follows:
\begin{eqnarray}\label{eq:funcional}
\hat J(y_0,u)=J(y(y_0,u),u),\quad \forall u\in {\Bbb U}_{\rm ad},\quad y_0\in {\Bbb R}^n,
\end{eqnarray}
where $y(y_0,u)$ solves \eqref{din_sis}. Then, the optimal control can be formulated as follows: for given $y_0\in {\Bbb R}^n$
we consider
$$
\min_{u\in {\Bbb U}_{\rm ad}} \hat J(y_0,u).
$$
The value function of the problem is defined as $v:{\Bbb R}^n\rightarrow {\Bbb R}$ as follows:
\begin{equation}\label{eq_v}
v(y)=\inf\left\{\hat J(y,u)\mid u\in {\Bbb U}_{\rm ad}\right\},\quad y\in {\Bbb R}^n.
\end{equation}
This function gives the best value for every initial condition, given the set of admissible controls $U_{\rm ad}$. It is characterized as the viscosity solution of the HJB equation corresponding to the infinite horizon {optimal control problem:}
\begin{equation}\label{HJB}
\lambda v(y)+\sup_{u\in U_{\rm ad}}\left\{-f(y,u)\cdot \nabla v(y)-g(y,u)\right\}=0,\quad y\in {\Bbb R}^n.
\end{equation}
{The solution of (\ref{HJB})} is unique for sufficiently large $\lambda$, $\lambda>\max(L_g,L_f)$, \cite{Bardi}. To construct the approximation scheme, as in \cite{falcone1}, let
us consider first a time discretization where $h$ is a strictly positive step size.
{We consider the following} semidiscrete scheme for \eqref{HJB}:
\begin{equation}\label{discrete_HJB}
v_h(y)=\min_{u\in U_{\rm ad}}\left\{(1-\lambda h)v_h(y+hf(y,u))+h g(y,u)\right\},\quad y\in {\Bbb R}^n.
\end{equation}
Under the assumptions \eqref{lip_f}, \eqref{infty_f}, \eqref{lip_g} and \eqref{cota_g} the function $v_h$ is Lipschitz-continuous and satisfies (see \cite[p. 473]{falcone2})
$$
|v_h(y)-v_h(\tilde y)|\le \frac{L_g}{\lambda -L_f}\|y-\tilde y\|_2,\quad \forall y,\tilde y\in \overline \Omega,\ h\in [0,1/\lambda).
$$
{The following  convergence result  for the semidiscrete approximation \cite[Theorem 2.3]{falcone1}  requires that
for $(y,\tilde y,u)\in {\Bbb R}^n\times {\Bbb R}^n\times U_{\rm ad}$
\begin{eqnarray}
 \|f(y+\tilde y,u)-2f(y,u)+f(y-\tilde y,u)\|_2&\le& C_f\|\tilde y\|_2^2,\label{semi_con_f}\\
 \|g(y+\tilde y,u)-2g(y,u)+g(y-\tilde y,u)\|_2&\le& C_g\|\tilde y\|_2^2.\label{semi_con_g}
 \end{eqnarray}
 }
 
\begin{Theorem}\label{th_semi} Let assumptions \eqref{lip_f}, \eqref{infty_f}, \eqref{invariance}, \eqref{lip_g}, \eqref{cota_g}, \eqref{semi_con_f} and \eqref{semi_con_g} hold and let $\lambda>\max(2L_g,L_f)$. Let $v$ and $v_h$ be the solutions
of \eqref{HJB} and \eqref{discrete_HJB}, respectively. Then, there exists a constant $C\ge 0$, that can be bounded explicitly, such that the
following bound holds
\begin{equation}\label{cota_semi}
\sup_{y\in {\Bbb R}^n}|v(y)-v_h(y)|\le Ch,\quad h\in [0,1/ \lambda).
\end{equation}
\end{Theorem}
Following \cite{falcone1}, \cite{falcone2} we introduce a fully discrete approximation to \eqref{HJB}. Let $\Omega$ a bounded polyhedron in $\Bbb R^n$ satisfying \eqref{invariance}. Let $\left\{S_j\right\}_{j=1}^{m_s}$ be a family of simplices which defines a regular triangulation of $\Omega$
$$
\overline \Omega=\bigcup_{j=1}^{m_s} S_j,\quad k=\max_{1\le j\le m_s}({\rm diam} \ S_j).
$$
We assume we have $n_s$ vertices/nodes $y_1,\ldots,y_{n_s}$ in the triangulation. Let $V^k$ be the space of piecewise affine functions from $\overline \Omega$ to ${\Bbb R}$ which are continuous in $\overline \Omega$ having constant gradients in the interior of any simplex $S_j$ of the triangulation. Then, a fully discrete scheme for the HJB equations is given by
\begin{equation}\label{fully_discrete}
v_{h,k}(y_i)=\min_{u\in {U}_{\rm ad}}\left\{(1-\lambda h)v_{h,k}(y_i+hf(y_i,u))+hg(y_i,u)\right\},
\end{equation}
for any vertex $y_i\in \overline \Omega$. Clearly, a solution to \eqref{discrete_HJB} satisfies \eqref{fully_discrete}. There exists a unique
solution of \eqref{fully_discrete} in the space $V^k$, see \cite[Theorem 1.1, Appendix A]{Bardi}.
The following result can be found in \cite[Corollary 2.4]{falcone1}, \cite[Theorem 1.3]{falcone2}, {it also requires
the semiconcavity assumptions \eqref{semi_con_f} and \eqref{semi_con_g}}.
\begin{Theorem} \label{th:mal} Let assumptions \eqref{lip_f}, \eqref{infty_f}, \eqref{invariance}, \eqref{lip_g}, \eqref{cota_g}, \eqref{semi_con_f} and \eqref{semi_con_g} hold.  Let $v$, $v_h$ and $v_{h,k}$ be the solutions
of \eqref{HJB}, \eqref{discrete_HJB} and \eqref{fully_discrete}, respectively. For $\lambda> L_f$ the following bound holds
$$
\|v_h-v_{h,k}\|_{C(\overline \Omega)}\le  \frac{L_g}{\lambda(\lambda-L_f)}\frac{k}{h},\quad h\in [0,1/\lambda).
$$
For $\lambda >\max(2L_g,L_f)$ the following bound holds
$$
\|v-v_{h,k}\|_{C(\overline \Omega)}\le  C h+ \frac{L_g}{\lambda(\lambda-L_f)}\frac{k}{h},\quad h\in [0,1/\lambda),
$$
where $C$ is the constant in \eqref{cota_semi}.
\end{Theorem}
As we can observe in the theorem the error bound of the fully discrete method deteriorates when the time step $h$ tends to zero. However, this behaviour of the method has not been observed in the literature. In next section, we improve the bound of Theorem~\ref{th:mal} proving that
the error behaves as $O(h+k)$ which gives first order both in space and time, as expected, for a method based on a first order discretization in time (Euler method) and a piecewise linear approximation in space.
\section{Optimal error bounds for the fully discrete approximations}
The key point to improve the above error bounds is to use a new characterization
for the function $v_{h,k}$. The characterization is based on the analogous characterization for the semi-discrete approximation $v_h$ that
can be found in \cite[Proposition 4.1, Chapter VI (appendix)]{Bardi}.
Let us define the space $\cal U$ of all sequences $\bu=\left\{u_0,u_1,\ldots,\right\}$ such that $u_i\in U_{\rm ad}$ and
for $\bu\in {\cal U}$ let us define the following fully discrete cost functional
\begin{eqnarray}\label{eq:funcional_fd}
\hat J_{h,k}(y,\bu)&:=&h\sum_{n=0}^\infty \delta_h^n I_{k}g(\hat y_n,u_n), \quad \delta_h=(1-\lambda h),\\
\hat y_{n+1}&=&\hat y_n+h I_{k} f(\hat y_n,u_n),\quad \hat y_0=y,\label{eq:euler}
\end{eqnarray}
where $I_{k}$ represents the linear interpolant of the function (respect to the $y$ variable) at the nodes $y_1,\ldots,y_{n_s}$.
\begin{Theorem}\label{th_prin}
The function $v_{h,k}\in V^k$ defined by \eqref{fully_discrete} satisfies
\begin{eqnarray}\label{eq:cha}
v_{h,k}(y)=\inf_{\bu\in {\cal U}} \hat J_{h,k}(y,\bu).
\end{eqnarray}
\end{Theorem}
\begin{proof}
The proof follows the argument of the proof of \cite[Proposition 4.1, Chapter VI (appendix)]{Bardi}.
First of all, let us observe that $\hat J_{h,k}(y,\bu)\in V^k$. To this end we observe that $I_{k}g(\hat y_n,u_n)\in V^k$ for $n=0,1,\ldots$.
In the case $n=0$, $I_{k}g(\hat y_0,u_0)=I_{k}g(y,u_0)$ which, by definition of $I_{k}$ belongs to {$V^k$}. For $n=1,$
$I_{k}g(\hat y_1,u_1)=I_{k}g(y+h I_{k}f(y,u_0),u_1)$. Then, since $I_{k}f(y,u_0)\in V^k$ the composed function
$I_{k}g(y+h I_{k}f(y,u_0),u_1)\in V^k$. The same argument can be applied for any of the terms and, as a conclusion, $\hat J_{h,k}(y,\bu)\in V^k$.
The series that defines $J_{h,k}$ is convergent under the same hypothesis that the series defining the semi-discrete functional
$v_h$ is convergent since  the interpolant heritages the  necesary properties of the interpolated function.

Let us prove now the thesis of the lemma. Let us define
$$
w_{h,k}(y)=\inf_{\bu\in {\cal U}} \hat J_{h,k}(y,\bu)\in V^k.
$$
Let us take $\bu\in {\cal U}$ and let us define $\overline \bu=\left\{u_1,u_2,\ldots,\right\}$.
We first observe that
\begin{eqnarray*}
\hat J_{h,k}(y_i,\bu)&=&hI_{k} g(y_i,u_0)+h\sum_{n=1}^\infty \delta_h^n I_{k} g(\hat y_n,u_n)\nonumber\\
&=&h g(y_i,u_0)+\delta_h \sum_{n=0}^\infty \delta_h^n I_{k} g(\hat y_{n+1},u_{n+1})\nonumber\\
&=&h g(y_i,u_0)+\delta_h\hat J_{h,k}(y_i+h f(y_i,u_0),\overline \bu),
\end{eqnarray*}
where we have applied that $I_{k}g(y_i,u_0)=g(y_i,u_0)$ and that $I_{k}f(y_i,u_0)=f(y_i,u_0)$ and the definition of $\hat J_{h,k}$.
Applying the definition of $w_{h,k}$ we get
$$
\hat J_{h,k}(y_i,\bu)\ge h g(y_i,u_0)+\delta_h w_{h,k}(y_i+h f(y_i,u_0)).
$$
And then,
$$
w_{h,k}(y_i)=\inf_{\bu\in {\cal U}}\hat J_{h,k}(y_i,\bu)\ge \inf_{\bu\in {\cal U}}
\left\{h g(y_i,u_0)+\delta_h w_{h,k}(y_i+h f(y_i,u_0))\right\}.
$$
Finally, since the right-hand side above depends only on $u_0$
\begin{equation}\label{eq:cha2}
\inf_{\bu\in {\cal U}}
\left\{h g(y_i,u_0)+w_{h,k}(y_i+h f(y_i,u_0))\right\}=
\inf_{u\in U_{\rm ad}}\left\{h g(y_i,u)+\delta_h w_{h,k}(y_i+h f(y_i,u))\right\}
\end{equation}
and then
$$
w_{h,k}(y_i)\ge \inf_{u\in U_{\rm ad}}\left\{h g(y_i,u)+\delta_h w_{h,k}(y_i+h f(y_i,u))\right\}.
$$
Now we take any $u_0\in U_{\rm ad}$ and any $\epsilon>0$ and denote by
$z=y_i+h f(y_i,u_0)$.
By definition of $w_{h,k}$, there exists $\bu^\epsilon\in {\cal U}$ such that
$$
w_{h,k}(z)+\epsilon\ge \hat J_{h,k}(z,\bu^\epsilon).
$$
Let us denote by
$$
\bu=\left\{u_0,\bu^\epsilon\right\}.
$$
Arguing as before, we get
$$
\hat J_{h,k}(y_i,\bu)=h g(y_i,u_0)+\delta_h \hat J_{h,k}(z,\bu^\epsilon).
$$
And then
$$
\hat J_{h,k}(y_i,\bu)\le h g(y_i,u_0)+\delta_h w_{h,k}(y_i+h f(y_i,u_0))+\epsilon.
$$
Arguing as before
$$
w_{h,k}(y_i)=\inf_{\bu\in {\cal U}}\hat J_{h,k}(y_i,\bu)\le \inf_{\bu\in {\cal U}}\left\{h g(y_i,u_0)+\delta_h w_{h,k}(y_i+h f(y_i,u_0))\right \}+\epsilon.
$$
And since, on the one hand, \eqref{eq:cha2} holds and, on the other, the above inequality is valid for any $\epsilon>0$ we get
$$
w_{h,k}(y_i)\le \inf_{u\in U_{\rm ad}}\left\{h g(y_i,u)+\delta_h w_{h,k}(y_i+h f(y_i,u))\right\}.
$$
Then
$$
w_{h,k}(y_i)= \inf_{u\in U_{\rm ad}}\left\{h g(y_i,u)+\delta_h w_{h,k}(y_i+h f(y_i,u))\right\},
$$
and $w_{h,k}\in V^k$ but since there is a unique function  in $V^k$ solving \eqref{fully_discrete} (see \cite[Theorem 1.1, Appendix A]{Bardi})
we get $w_{h,k}=v_{h,k}$ which finishes the proof of the theorem.
\end{proof}
{
In the rest of the paper we apply Theorem \ref{th_prin} with two different scenarios. In the first subsection, assuming enough 
regularity for the controls, we can get the full first order in time and space. In the second one, following \cite{boba_etal}, 
we make a proof using piecewise constants controls.  However, adapting \cite{boba_etal}, which is written in the context of finite horizon to our infinite horizon problem, we loose the full order in time in the
rate of convergence.}
\subsection{Error analysis assuming some regularity for the controls
}
{
Let us denote by
\begin{equation}\label{la_Mu}
M_u:=\max_n\max_{s\in[nh,(n+1)h)}\|u(s)-u(t_n)\|_2.
\end{equation}
In the proof of next lemma we assume that the following condition holds for the controls
\begin{equation}\label{conver_con}
\lim_{h\rightarrow 0}M_u=0.
\end{equation}
Let us observe that assuming the controls are uniformly continuous condition \eqref{conver_con} always holds.}
\begin{lema}\label{le:con}
Let $\hat J$ and $\hat J_{h,k}$ be the functionals defined in \eqref{eq:funcional} and \eqref{eq:funcional_fd} respectively.
Assume conditions  \eqref{lip_f}, \eqref{lip_f2}, \eqref{infty_f}, \eqref{lip_g}, \eqref{lip_g2}, \eqref{cota_g} and \eqref{conver_con} hold. Then
\begin{eqnarray}\label{eq_lim0}
\lim_{h\rightarrow 0, k\rightarrow 0} |\hat J(y_0,u)-\hat J_{h,k}(y_0,\bu)|=0,\quad y_0\in {\Bbb R}^n,
\end{eqnarray}
where $u\in {\Bbb U}_{\rm ad}$ and $\bu=\left\{u_0,u_1,\ldots,\right\}=\left\{u(t_0),u(t_1),\ldots,\right\}$, $t_i=ih$, $i=0,1,\ldots,$.
\end{lema}
\begin{proof}
We argue similarly as in \cite[Lemma 1.2, Chapter VI]{Bardi}. Let $y(t)$ be the solution of \eqref{din_sis} and let us denote by $\tilde y(t)=\hat y_k$,
$k=[t/h]$ where $\hat y_k$ is the solution of \eqref{eq:euler} with $\hat y_0=y_0$. Let us denote by
\begin{equation}\label{paluego}
\overline u(t)=u_k=u(t_k),\quad t\in[kh,(k+1)h).
\end{equation}
Then, $\tilde y$ can be expressed as
$$
\tilde y(t)=y_0+\int_0^{[t/h]h}I_{k}f(\tilde y(s),\overline u(s))~ds.
$$
And,
\begin{eqnarray*}
y(t)-\tilde y(t)=\int_0^{[t/h]h}\left(f(y(s),u(s))-I_{k}f(\tilde y(s),\overline u(s))\right)~ds
+\int_{[t/h]h}^t f(y(s),u(s))~ds.
\end{eqnarray*}
From the above equation, applying \eqref{infty_f}, we get
\begin{eqnarray}\label{eq:infty1}
\|y(t)-\tilde y(t)\|_\infty\le \int_0^{[t/h]h}\|f(y(s),u(s))-I_{k}f(\tilde y(s),\overline u(s))\|_\infty~ds +M_f h.
\end{eqnarray}
Let us bound now the term inside the integral.
Adding and subtracting terms we get
\begin{eqnarray}
&&\|f(y(s),u(s))-I_{k}f(\tilde y(s),\overline u(s))\|_\infty\le
\|f(y(s),u(s))-f( y(s),\overline u(s))\|_\infty\label{eq:decom}\\
&&\quad +\|f(y(s),\overline u(s))-I_{k}f(y(s),\overline u(s))\|_\infty+\|I_{k}f(y(s),\overline u(s))-I_{k}f(\tilde y(s),\overline u(s))\|_\infty.\nonumber
\end{eqnarray}
For the first term on the right-hand side of \eqref{eq:decom} using \eqref{lip_f2} and \eqref{lip_u}
and assuming $s\in [kh,(k+1)h)$ we get
\begin{eqnarray}\label{eq:primer}
\|f(y(s),u(s))-f( y(s),\overline u(s))\|_\infty\le L_f \|u(s)-\overline u(s)\|_2
&=&L_f\|u(s)-u(t_k)\|_2\nonumber\\
&\le& L_f {M_u},
\end{eqnarray}
{for $M_u$ defined in \eqref{la_Mu}}.
{Let us observe that assuming condition \eqref{conver_con} holds the above term goes to zero as
$h$ goes to zero.}

To bound the second term on the right-hand side of \eqref{eq:decom} arguing as in \cite{Alla_Falcone_Volkwein} we observe that
for any $y\in  \overline \Omega$ there exists an index $l$ with $y\in \overline S_l\subset \overline \Omega$. Let
us denote by $J_l$ the index subset such that $y_i\in S_l$ for $i\in J_l$. Writing
$$
y=\sum_{i=1}^{n_S} \mu_i y_i,\quad 0\le \mu_i\le 1,\quad \sum_{i=1}^{n_S} \mu_i=1,
$$
it is clear that $\mu_i=0$ holds for any $i\not \in J_l$. Now, we observe that for any $u\in U_{\rm ad}$ and $j=1,\ldots,n$,
applying \eqref{lip_f} we get
\begin{eqnarray}\label{needed}
|f_j(y,u)-I_{k}f_j(y,u)|&=&|\sum_{i=1}^{n_S}\mu_i f_j(y,u)-\sum_{i=1}^{n_S}\mu_i I_{k}f_j(y_i,u)|\nonumber\\
&=&|\sum_{i\in J_l}\mu_i(f_j(y,u)-f_j(y_i,u)|\nonumber\\
&\le& \sum_{i\in J_l}\mu_i L_f \|y-y_i\|_2\le L_f k,
\end{eqnarray}
where in the last inequality we have applied $\|y-y_i\|_2\le k$, for $i\in J_l$. From the above inequality we get for the second term
on the right-hand side of \eqref{eq:decom}
\begin{eqnarray}\label{eq:segundo}
\|f(y(s),\overline u(s))-I_kf( y(s),\overline u(s))\|_\infty\le L_f k.
\end{eqnarray}
For the third term on the right-hand side of \eqref{eq:decom} we observe that the difference of the interpolation operator evaluated at
two different points can be bounded in terms of the constant gradient of the interpolant in the element to which those points belong times
the difference of them, i.e.,
$$
I_{k}f_j(y,u)-I_{k}f_j(\tilde y,u)=\nabla I_{k} f_j(\tilde y,u)\cdot (y-\tilde y)\le \|\nabla I_{k} f_j(\tilde y,u)\|_2\|y-\tilde y\|_2.
$$
Moreover, $\nabla I_{k}f_k$ can be bounded in terms of the lipschitz constant of $f$, $L_f$, more precisely,  $\|\nabla I_{k} f_j(\tilde y,u)\|_2\le C\sqrt{n}L_f$. Then,
$$
|I_{k}f_j(y,u)-I_{k}f_j(\tilde y,u)|\le C L_f \sqrt{n}\|y-\tilde y\|_2.
$$
As a consequence, for
the third term on the right-hand side of \eqref{eq:decom} we get
\begin{eqnarray}\label{eq:tercero}
\|I_{k}f(y(s),\overline u(s))-I_{k}f(\tilde y(s),\overline u(s))\|_\infty&\le& C \sqrt{n}L_f\|y(s)-\tilde y(s)\|_2\nonumber\\
&\le& C nL_f\|y(s)-\tilde y(s)\|_\infty.
\end{eqnarray}
Inserting \eqref{eq:primer}, \eqref{eq:segundo} and \eqref{eq:tercero} into \eqref{eq:infty1} we get for
\begin{equation}\label{overL}
\overline L=C nL_f,
\end{equation}
\begin{eqnarray*}
\|y(t)-\tilde y(t)\|_\infty\le \overline L \int_0^t \|y(s)-\tilde y(s)\|_\infty~ds+ tL_f(M_u+k)+M_f h.
\end{eqnarray*}
Applying Gronwall's lemma we obtain
\begin{eqnarray}\label{ymenosytilde}
\|y(t)-\tilde y(t)\|_\infty\le \frac{e^{\overline L t}}{\overline L}\left(tL_f(M_u+k)+M_f h\right).
\end{eqnarray}
Applying \eqref{cota_g} and taking into account that $\|I_{k}g(y,u)\|_\infty\le \|g(y,u)\|_\infty$ (where $\|\cdot\|_\infty$ refers to the
$L^\infty$ norm respect to the first argument) it is easy to check that
\begin{eqnarray}\label{eq:tb}
|\hat J(y_0,u)-J_{h,k}(y_0,\bu)|\le X_1+X_2+X_3,
\end{eqnarray}
with
\begin{eqnarray*}
X_1&=&\left|h\sum_{n=0}^{[T/h]-1} \delta_h^n I_{k}g(\hat y_n,u_n)-\int_0^T g(y(s),u(s))e^{-\lambda s}~ds\right|,\\
X_2&=&\left|h \sum_{n=[T/h]}^\infty M_g \delta_h^n\right|,\quad X_3=\int_T^\infty M_g e^{-\lambda s}~ds,
\end{eqnarray*}
and $T>0$ arbitrary. Now, we will estimate $X_i$. It is easy to see that
$$
X_2+X_3\le M_g h \frac{\delta_h^{[T/h]}}{\lambda h}+M_g \frac{e^{-\lambda T}}{\lambda}.
$$
Since $\delta_h^{[T/h]}\rightarrow e^{-\lambda T}$ when $h\rightarrow 0$ then for any $\epsilon>0$ there exists $\overline h=\overline h(\epsilon, \lambda, M_g)>0$ and $\overline T=\overline T(\epsilon, \lambda, M_g)>0$ such that
\begin{eqnarray}\label{eq:tb2}
X_2+X_3\le \epsilon,\quad {\rm for}\ {\rm all}\ 0<h\le \overline h,\quad T\ge \overline T.
\end{eqnarray}
In the argument that follows we fix $T=\overline T$.
We observe that
$$
X_1=\left|\int_0^{[T/h]h}I_{k}g(\tilde y(s),\overline u(s))\delta_h^{[s/h]}~ds-\int_0^T g(y(s),u(s))e^{-\lambda s}~ds\right|.
$$
Then, we can write
\begin{eqnarray}\label{cota_x1}
X_1&\le& X_{1,1}+X_{1,2}+X_{1,3}+X_{1,4}+X_{1,5}\\
&:=& \left|\int_0^{[T/h]h} I_{k}g(\tilde y(s),\overline u(s))(\delta_h^{[s/h]}- e^{-\lambda s})~ds\right|\nonumber\\
&&\quad+\left|\int_0^{[T/h]h} I_{k} \left(g(\tilde y(s),\overline u(s))-I_{k}g(y(s),\overline u(s))\right)e^{-\lambda s}~ds\right|
\nonumber\\
&&\quad+\left|\int_0^{[T/h]h} \left(I_{k}g(y(s),\overline u(s))-g(y(s),\overline u(s))\right)e^{-\lambda s}~ds\right|\nonumber\\
&&\quad +\left|\int_0^{[T/h]h} \left(g(y(s),\overline u(s))-g(y(s),u(s))\right)e^{-\lambda s}~ds\right|\nonumber\\
&&\quad +\left|\int_{[T/h]h}^T g(y(s),u(s))e^{-\lambda s}~ds\right|.\nonumber
\end{eqnarray}
We will bound the terms on the right-hand side of \eqref{cota_x1}. To bound the first term we will apply again $\|I_{k}g(y,u)\|_\infty\le \|g(y,u)\|_\infty$ and \eqref{cota_g} to obtain
\begin{eqnarray*}\label{cota_x1_1}
&&X_{1,1}=\left|\int_0^{[T/h]h} I_{k}g(\tilde y(s),\overline u(s))(\delta_h^{[s/h]}- e^{-\lambda s})~ds\right|
\nonumber\\
&&\quad\le \int_0^T |I_{k}g(\tilde y(s),\overline u(s))||\delta_h^{[s/h]}- e^{-\lambda s}|~ds
\le M_g\int_0^T |\delta_h^{[s/h]}- e^{-\lambda s}|~ds.
\end{eqnarray*}
Now we write $\delta_h^{[s/h]}=e^{-\lambda \theta [s/h]h}$, for $\theta=-\log(\delta_h)/(\lambda h)$. Applying the mean value theorem to the function $e^{-\lambda s}$ and taking into account that since $[s/h]h\le s\le [s/h]h+h$ then $|s-\theta [s/h]h|\le (\theta-1)T+\theta h$
and that $\theta\rightarrow 1$ when $h\rightarrow 0$ then we get
\begin{eqnarray}\label{cota_x1_1}
X_{1,1} \le M_g T \lambda((\theta-1)T+\theta h)\le \epsilon,
\end{eqnarray}
for $h\le \overline h$, with $\overline h=\overline h(\epsilon, \overline T, \lambda, M_g)$.

To bound the next term we argue as in \eqref{eq:tercero} and use \eqref{ymenosytilde} to get
\begin{eqnarray*}
X_{1,2}\le\frac{  C {n}L_g}{\overline L}  \int_0^T  {e^{\overline L s}}\left(sL_f(M_u+k)+M_f h\right)
e^{-\lambda s}~ds\le C_1 h+C_2 k+{C_3 M_u},
\end{eqnarray*}
where
$$
{C_1=\frac{  C {n}L_g}{\overline L}\int_0^T e^{(\overline L-\lambda)s} M_f~ds},
\quad C_2= \frac{  C\sqrt{n}L_g}{\overline L}\int_0^T e^{(\overline L-\lambda)s} sL_f~ds,
$$
and
$$
{C_3=\frac{  C {n}L_g}{\overline L}\int_0^T e^{(\overline L-\lambda)s} sL_f ~ds}.
$$
Then, {assuming condition \eqref{conver_con} holds to assure convergence for the third term in the error bound of $X_{1,2}$}
\begin{eqnarray}\label{cota_x1_2}
X_{1,2}\le \epsilon,\quad h\le \overline h, \ k\le \overline k,
\end{eqnarray}
where
$\overline h=\overline h(\epsilon, \overline T,\lambda, C,\overline L, n, L_g, L_f, {M_u}, M_f)$,
$\overline k=\overline k(\epsilon, \overline T,\lambda, C,\overline L, n, L_g, L_f)$.
For the third term, arguing as in \eqref{needed} we get
\begin{eqnarray}\label{cota_x1_3}
X_{1,3}\le L_g k \int_0^T e^{-\lambda s}~ds\le \epsilon,\quad k\le \overline k=\overline k(\epsilon,\overline T, \lambda,L_g).
\end{eqnarray}
To bound the forth term we apply \eqref{lip_g} and \eqref{lip_u} to get
$$
|g(y(s),\overline u(s))-g(y(s),u(s))|\le L_g  \|u(s)-u(t_k)\|_2\le L_g {M_u},\quad s\in[k h,(k+1)h).
$$
And then, {assuming again condition \eqref{conver_con} holds}
\begin{eqnarray}\label{cota_x1_4}
X_{1,4}\le L_g {M_u}\int_0^T e^{-\lambda s}~ds\le \epsilon,\quad h\le \overline h=\overline h(\epsilon,\overline T, \lambda,L_g,{M_u}).
\end{eqnarray}
Finally, for the last term on the right-hand side of \eqref{cota_x1}, applying \eqref{cota_g},
\begin{eqnarray}\label{cota_x1_5}
X_{1,5}\le M_g h\le \epsilon,\quad h\le \overline h=\overline h(\epsilon,M_g).
\end{eqnarray}
Inserting \eqref{cota_x1_1}, \eqref{cota_x1_2}, \eqref{cota_x1_3}, \eqref{cota_x1_4} and \eqref{cota_x1_5}
into \eqref{cota_x1} and taking into account \eqref{eq:tb} and \eqref{eq:tb2} we conclude \eqref{eq_lim0}.
\end{proof}
\begin{lema}\label{le:con_hk}
Let $\hat J$ and $\hat J_{h,k}$ be the functionals defined in \eqref{eq:funcional} and \eqref{eq:funcional_fd} respectively.
Assume conditions  \eqref{lip_f}, \eqref{lip_f2}, \eqref{infty_f}, \eqref{lip_g}, \eqref{lip_g2}, \eqref{cota_g} hold.
Assume $\lambda>\overline L$ with $\overline L$ the constant in \eqref{overL}.
Then, for $0\le h\le 1/(2\lambda)$ there exist  positive constants

\noindent ${C_1=
C_1(\lambda,M_f,M_g,L_f,L_g)}$ and ${C_2=C_2(\lambda,L_f,L_g)}$ such that
\begin{eqnarray}\label{eq_lim0_hk}
 |\hat J(y_0,u)-\hat J_{h,k}(y_0,\bu)|\le C_1(h+k)+{C_2 M_u},\quad y_0\in {\Bbb R}^n,
\end{eqnarray}
where $u\in {\Bbb U}_{\rm ad}$ and $\bu=\left\{u_0,u_1,\ldots,\right\}=\left\{u(t_0),u(t_1),\ldots,\right\}$, $t_i=ih$, $i=0,1,\ldots,$
{and $M_u$ is defined in \eqref{la_Mu}}.
\end{lema}
\begin{proof}
We argue as in \cite[Lemma 1.2, Chapter VI]{Bardi}. Let $y(t)$ be the solution of \eqref{din_sis} and let us denote by $\tilde y(t)=\hat y_k$,
$k=[t/h]$ where $\hat y_k$ is the solution of \eqref{eq:euler} with $\hat y_0=y_0$. Let us denote by
$$
\overline u(t)=u_k=u(t_k),\quad t\in[kh,(k+1)h).
$$
We first observe that
$$
|\hat J(y_0,u)-\hat J_{h,k}(y_0,\bu)|\le X+Y,
$$
where
\begin{eqnarray*}
X&=&\int_0^\infty |g(y(s),u(s))-I_{k}g(\tilde y(s),\overline u(s))|e^{-\lambda s}~ds,\\
Y&=&\int_0^\infty |I_{k}g(\tilde y(s),\overline u(s))||e^{-\lambda s}-e^{-\lambda \theta[s/h]h}|~ds,
\end{eqnarray*}
and, as in Lemma~\ref{le:con}, $\theta=-\log(\delta_h)/(\lambda h)$. To bound $X$ we
decompose
\begin{eqnarray*}
|g(y(s),u(s))-I_{k}g(\tilde y(s),\overline u(s))|&\le& |g(y(s),u(s))-g(y(s),\overline u(s))|\nonumber\\
&&\quad+|g(y(s),\overline u(s))-I_{k}g(y(s),\overline u(s))|\nonumber\\
&&\quad+|I_{k}g(y(s),\overline u(s))-I_{k}g(\tilde y(s),\overline u(s))|.
\end{eqnarray*}
Arguing as in Lemma~\ref{le:con} we get
\begin{eqnarray*}
|g(y(s),u(s))-I_{k}g(\tilde y(s),\overline u(s))|&\le&
L_g {M_u}+ L_g k +C n  L_g \|y(s)-\tilde y(s)\|_\infty
\end{eqnarray*}
and then applying \eqref{ymenosytilde} we obtain
\begin{eqnarray*}
X\le \frac{1}{\lambda}\left(L_g {M_u}+ L_g k\right)+C \frac{n L_g }{\overline L}\int_0^\infty \left(sL_f({M_u}+k)+M_f h\right)e^{(\overline L-\lambda) s}~ds.
\end{eqnarray*}
Then, for $\lambda >\overline L$ there exist constants $C_1=C_1(\lambda,M_f,L_f,L_g,L_u)$
and $C_2(\lambda,L_f,L_g)$ such that
$$
X\le C_1 (h+k)+{C_2M_u}.
$$
To bound $Y$ we first observe that, arguing as before, $|I_{k}g(\tilde y(s),\overline u(s))|\le M_g$ and then
$$
Y\le M_g \int_0^\infty |e^{-\lambda s}-e^{-\lambda \theta[s/h]h}|~ds.
$$
Applying the mean value theorem
$$
Y\le M_g\int_0^\infty \lambda \max\left\{e^{-\lambda s},e^{-\lambda \theta [s/h]h} \right\}|\lambda s-\lambda \theta [s/h]h|~ds.
$$
Now, since $|s-\theta [s/h ]h|\le (\theta-1)s+\theta h$ we get
$$
Y\le M_g\lambda^2 e^{\theta \lambda h}\int_0^\infty e^{-\lambda s}((\theta -1)s+\theta h)~ds,
$$
and since both
$$
\int_0^\infty s e^{-\lambda s}~ds,\quad \int_0^\infty e^{-\lambda s}~ds,
$$
are bounded and $\lim_{h\rightarrow 0}(\theta -1)h=\lambda/2$ (for $0\le h\le 1/(2\lambda)$ the function $(\theta -1)/h$ is
an increasing function bounded by its value at $h=1/(2\lambda)$, i.e., $(\theta-1)/h\le 2\lambda(2\log(2)-1)$, we conclude
$$
Y\le C h,\quad C=C(M_g,\lambda)>0.
$$
\end{proof}
\begin{Theorem}\label{th_4}
Assume conditions  \eqref{lip_f}, \eqref{lip_f2}, \eqref{infty_f}, \eqref{lip_g}, \eqref{lip_g2} and \eqref{cota_g}  hold.
Assume $\lambda>\overline L$ with $\overline L$ the constant in \eqref{overL}. Then, for $0\le h\le 1/(2\lambda)$ there exists positive constants ${C_1=
C_1(\lambda,M_f,M_g,L_f,L_g)}$ and  ${C_2=
C_2(\lambda,L_f,L_g)}$ such that
\begin{equation}\label{cota_buena}
|v(y)-v_{h,k}(y)|\le C_1(h+k)+{C_2M_u},\quad y\in {\Bbb R}^n,
\end{equation}
{and $M_u$ is defined in \eqref{la_Mu}}.
\end{Theorem}
\begin{proof}
In view of \eqref{eq:cha}
let us denote by $\overline \bu\in {\cal U}$ a control giving the minimum
$$
v_{h,k}(y)=\hat J_{h,k}(y,\overline \bu).
$$
Then
$$
v(y)-v_{h,k}(y)\le \hat J(y,\overline u)-J_{h,k}(y,\overline \bu),
$$
where $\overline u\in{\Bbb U}_{\rm ad}$ such that $\overline u(t_i)=\overline u_i$, $\overline \bu=\left\{\overline u_0,\overline u_1,\ldots,\right\}$.
Applying \eqref{eq_lim0_hk}, there exist  positive constants ${C_1=
C(\lambda,M_f,M_g,L_f,L_g)}$ and ${C_2=
C(\lambda,L_f,L_g)}$ such that
\begin{equation}\label{puessi}
v(y)-v_{h,k}(y)\le C_1(h+k)+{C_2M_u}.
\end{equation}
Now, let us denote by $\underline u \in{\Bbb U}_{\rm ad}$ the control giving the minimum in \eqref{eq_v} and let us denote by
$
\underline \bu=\left\{\underline u(t_0),\underline u(t_1),\ldots,\right\}.
$
Then, arguing as before there exist  positive constants ${C_1=
C(\lambda,M_f,M_g,L_f,L_g)}$ and ${C_2=
C(\lambda,L_f,L_g)}$ such that
$$
v_{h,k}(y)-v(y)\le  \hat J_{h,k}(y,\underline \bu)-\hat J(y,\underline u)\le C_1 (h+k)+{C_2M_u},
$$
which together with \eqref{puessi} implies \eqref{cota_buena}.
\end{proof}
{
\begin{remark}
We observe that Theorem \ref{th_4} gives a bound for the error without assuming any regularity on the controls. The error bound \eqref{cota_buena}
has two terms. The first one is always first order convergent both in time and space. The second one depends on the size of
$M_u$ defined in \eqref{la_Mu}, which depends on the regularity of the controls. As it is always the case when one
applies numerical methods, some regularity is required to achieve the best rate of convergence as possible. In
Theorems \ref{th_5} and \ref{th_6} below we get as  a consequence of Theorem 4 two possible scenarios. In Theorem \ref{th_5}, assuming condition \eqref{conver_con} holds
for the controls (which is true for uniformly continuous controls) we prove convergence. In Theorem \ref{th_6}, assuming that the controls are Lipschitz-continuous, \eqref{lip_u}, we prove that $M_u$ in \eqref{la_Mu} behaves as $h$, so that, in \eqref{cota_buena}, the optimal rate of convergence
of order one both in time and space is recovered. Although in a concrete problem one could not have enough regularity for the controls to achieve
first order of convergence in time (observe that the rate of convergence is always one in space) the error bound \eqref{cota_buena} can always be applied. This error bound allow us to identify the different sources of the error in the method and consequently is able to
explain the behavior of the method. As a conclusion, we can discard a behaviour for the rate of convergence as $O(k/h)$, as the existing bounds in the literature had predicted.
\end{remark}}
{In the following theorem we deduce convergence as $h$ goes to zero assuming condition \eqref{conver_con} holds.
\begin{Theorem}\label{th_5}
Assume conditions  \eqref{lip_f}, \eqref{lip_f2}, \eqref{infty_f}, \eqref{lip_g}, \eqref{lip_g2}, \eqref{cota_g}
and \eqref{conver_con} hold.
Assume $\lambda>\overline L$ with $\overline L$ the constant in \eqref{overL}. Then,
$$
\lim_{h\rightarrow 0,k\rightarrow 0}|v(y)-v_{h,k}(y)|=0,\quad y\in {\Bbb R}^n.
$$
\end{Theorem}
\begin{proof}
The conclusion is obtained as a corollary of Theorem 4 applying \eqref{conver_con} to
bound the second term in \eqref{cota_buena}.
\end{proof}
}
{If we assume that the controls are Lipschitz-continuous; i.e., there exists a positive constant $L_u>0$ such that
\begin{eqnarray}\label{lip_u}
\|u(t)-u(s)\|_2\le L_u |t-s|.
\end{eqnarray}
we can prove first order of convergence both in time and space}.
{
\begin{Theorem}\label{th_6}
Assume conditions  \eqref{lip_f}, \eqref{lip_f2}, \eqref{infty_f}, \eqref{lip_g}, \eqref{lip_g2}, \eqref{cota_g} and \eqref{lip_u} hold.
Assume $\lambda>\overline L$ with $\overline L$ the constant in \eqref{overL}. Then, for $0\le h\le 1/(2\lambda)$ there exist positive constants ${C_1=
C_1(\lambda,M_f,M_g,L_f,L_g)}$ and  ${C_2=
C_2(\lambda,L_f,L_g,L_u)}$ such that
$$
|v(y)-v_{h,k}(y)|\le C_1(h+k)+{C_2h},\quad y\in {\Bbb R}^n.
$$
\end{Theorem}
\begin{proof}
The conclusion is obtained as a corollary of Theorem 4 applying \eqref{lip_u} to
bound the second term in \eqref{cota_buena}.
\end{proof}
}
\begin{remark}
{
The regularity requirements on the controls can still be weakened. Let us assume that the controls have a finite number of jump discontinuities:
$t_1^*< t_2^*<\ldots <t_l^*,$. For a fixed $h$ let us denote by $I_j^*=[m_jh,(m_j+1)h)$, $m_j\in {\Bbb N},$  the interval such that
$t_j^*\in I_j^*$, $j=1,\ldots,l$. Then, the arguments of Lemma \ref{le:con_hk} can be adapted to get instead of \eqref{eq_lim0_hk}
the following bound
\begin{eqnarray}\label{eq_lim0_hk_new}
 |\hat J(y_0,u)-\hat J_{h,k}(y_0,\bu)|\le C_1(h+k)+C_2 h+{C_3 M_u^*},\quad y_0\in {\Bbb R}^n,
\end{eqnarray}
where $C_1=C_1(\lambda,M_f,M_g,L_f,L_g)$, {$C_2=C_2(\lambda,L_f,L_g,M_s)$} and $C_3=C_3(\lambda,L_f,L_g)$ and
\begin{eqnarray*}
M_u^*&:=&\max_n\max_{s\in[nh,(n+1)h), s\not\in I}\|u(s)-u(t_n)\|_2,\quad I=I_1^*\cup\ldots \cup I_l^*,\\
M_s&:=&\max_{1\le j\le l}\max_{s\in I_j^*}\|u(s)-u(t_{m_j})\|_2.
\end{eqnarray*}
The idea is to use the additive property of integrals to isolate those corresponding to intervals $I_j^*$, $j=1,\ldots,l$.
To bound the terms involving integrals over $I_j^*$ one can apply the boundedness of the integrand together with the
fact that the diameter of $I_j^*$ is equal $h$ for $j=1,\ldots,l$. The union of the bounds corresponding to integrals
over $I_j^*$ gives rise to the second term on the right-hand-side of \eqref{eq_lim0_hk_new}.}

{Accordingly, instead of \eqref{cota_buena} in Theorem \ref{th_4}, one can prove
\begin{equation}\label{cota_mejor}
|v(y)-v_{h,k}(y)|\le C_1(h+k)+C_2 h+{C_3 M_u^*},\quad y\in {\Bbb R}^n,
\end{equation}
with $C_1, C_2$ and $C_3$ the constants in \eqref{eq_lim0_hk_new}.}

{From \eqref{cota_mejor}, and depending on the regularity of the controls, one gets either convergence or convergence of order one, arguing as in Theorems \ref{th_5} and \ref{th_6}
but with the assumptions on the controls restricted to the finite number of intervals: $$[0,t_1^*),\ldots [t_j^*,t_{j+1}^*),\ldots,[t_l^*,\infty).$$ Moreover, assuming that the controls are H$\ddot{\rm o}$lder continuous over those intervals  of order $\alpha$, for any $0<\alpha<1$, one
gets an error bound in time of size $O(h^\alpha)$, applying \eqref{cota_mejor}.
}
\end{remark}
\subsection{Error analysis arguing with piecewise constants controls}
In this section we adapt the error analysis in \cite{boba_etal} for finite horizon problems to our context of infinite horizon problems
to get a weaker result for the rate of convergence but weakening also the regularity assumptions over the controls.

Let us denote by
$$
{\Bbb U}^{pc}_{\rm ad}=\left\{u\in {\Bbb U}\mid u(t)=u_k\in U_{\rm ad}, \  t\in[t_k,t_{k+1})\right\},
$$
with $u_k$ constant.
Let us observe that we can consider the continuous problem for controls in ${\Bbb U}^{pc}_{\rm ad}$. The
following lemmas are a direct consequence of Lemmas \ref{le:con} and \ref{le:con_hk}
\begin{lema}\label{le:con_cons}
Let $\hat J$ and $\hat J_{h,k}$ be the functionals defined in \eqref{eq:funcional} and \eqref{eq:funcional_fd} respectively.
Assume conditions  \eqref{lip_f}, \eqref{lip_f2}, \eqref{infty_f}, \eqref{lip_g}, \eqref{lip_g2}, \eqref{cota_g}. Then
\begin{eqnarray}\label{eq_lim0_cons}
\lim_{h\rightarrow 0, k\rightarrow 0} |\hat J(y_0,u)-\hat J_{h,k}(y_0,\bu)|=0,\quad y_0\in {\Bbb R}^n,
\end{eqnarray}
where $u\in {\Bbb U}^{pc}_{\rm ad}$ and $\bu=\left\{u_0,u_1,\ldots,\right\}$ with $u_k=u(t),$ $t\in[t_k,t_{k+1})$.
\end{lema}
\begin{lema}\label{le:con_hk_cons}
Let $\hat J$ and $\hat J_{h,k}$ be the functionals defined in \eqref{eq:funcional} and \eqref{eq:funcional_fd} respectively.
Assume conditions  \eqref{lip_f}, \eqref{lip_f2}, \eqref{infty_f}, \eqref{lip_g}, \eqref{lip_g2}, \eqref{cota_g} hold.
Assume $\lambda>\overline L$ with $\overline L$ the constant in \eqref{overL}.
Then, for $0\le h\le 1/(2\lambda)$ there exists a  positive constant
$C_1=C_1(\lambda,M_f,M_g,L_f,L_g)$ such that
\begin{eqnarray}\label{eq_lim0_hk_cons}
 |\hat J(y_0,u)-\hat J_{h,k}(y_0,\bu)|\le C_1(h+k),\quad y_0\in {\Bbb R}^n,
\end{eqnarray}
where $u\in {\Bbb U}_{\rm ad}^{pc}$ and $\bu=\left\{u_0,u_1,\ldots,\right\}$ with $u_k=u(t),$ $t\in[t_k,t_{k+1})$.
\end{lema}
The proof of Lemmas \ref{le:con_cons} and \ref{le:con_hk_cons} is obtained taking in the proofs of Lemmas \ref{le:con} and \ref{le:con_hk}
$M_u=0$ since for piecewise controls it holds, see \eqref{paluego},
$$
\overline u(t)=u_k=u(t),\ \forall t\in[kh,(k+1)h).
$$
The following theorem is analogous to Theorem \ref{th_4}. For the proof we need to assume an additional convexity assumption, see \cite[(A4)]{boba_etal},
\begin{itemize}
\item (CA) For every $y\in {\Bbb R}^n$,
\begin{eqnarray*}
\left\{f(y,u), g(y,u),\quad  u\in U_{\rm ad}\right\}
\end{eqnarray*}
is a convex subset of ${\Bbb R}^{n+1}$.
\end{itemize}
\begin{Theorem}\label{th_4_cons}
Assume conditions  \eqref{lip_f}, \eqref{lip_f2}, \eqref{infty_f}, \eqref{lip_g}, \eqref{lip_g2}, \eqref{cota_g} and {\rm(CA)}  hold.
Assume $\lambda>\overline L$ with $\overline L$ the constant in \eqref{overL}. Then, for $0\le h\le 1/(2\lambda)$ there exist positive constants $C_1=
C_1(\lambda,M_f,M_g,L_f,L_g)$ and  $C_2=
C_2(\lambda,M_f,M_g,L_f,L_g)$ such that for $y\in {\Bbb R}^n$
\begin{equation}\label{cota_buena_cons}
|v(y)-v_{h,k}(y)|\le C_1(h+k)+C_2 \frac{1}{(1+\beta)^2\lambda^2}(\log(h))^2h^{\frac{1}{1+\beta}},\quad \beta=\frac{\sqrt{n}L_f}{\lambda}. 
\end{equation}
\end{Theorem}
\begin{proof}
In view of \eqref{eq:cha}
let us denote by $\overline \bu\in {\cal U}$ a control giving the minimum
$$
v_{h,k}(y)=\hat J_{h,k}(y,\overline \bu).
$$
Then
$$
v(y)-v_{h,k}(y)\le \hat J(y,\overline u)-J_{h,k}(y,\overline \bu),
$$
where $\overline u\in{\Bbb U}_{\rm ad}^{pc}$ such that $\overline u(t)=\overline u_i$, $t\in[t_i,t_{i+1})$.
 Applying \eqref{eq_lim0_hk_cons}, there exists a  positive constant $C_1=
C(\lambda,M_f,M_g,L_f,L_g)$  such that
\begin{equation}\label{puessi}
v(y)-v_{h,k}(y)\le C_1(h+k).
\end{equation}
Now, let us denote by $\underline u \in{\Bbb U}_{\rm ad}$ the control giving the minimum in \eqref{eq_v}
so that 
\begin{equation}\label{fal_0}
v(y)=\hat J(y,\underline u)=\int_0^\infty g(y(t),\underline u(t))e^{-\lambda t}dt.
\end{equation} 
The following argument is taken from \cite[Appendix B]{boba_etal}. 

For any $t_k$ we can write
$$
y(t)=y(t_k)+\int_{t_k}^t f(y(s),\underline u(s))ds.
$$
Applying \eqref{infty_f}
$$
\|y(t)-y(t_k)\|_\infty \le M_f h.
$$
Then, for any $t\in [t_k,t_{k+1}]$, using the above inequality and \eqref{lip_f} we obtain
$$
\left\|\int_{t_k}^t f(y(s),\underline u(s))-f(y(t_k),\underline u(s)) ds\right\|_\infty\le \sqrt{n}L_f M_f h^2.
$$
As a consequence, we get
\begin{equation}\label{fal_1}
y(t)=y(t_k)+\int_{t_k}^t f(y(t_k),\underline u(s))ds+\epsilon_k,\quad \|\epsilon_k\|_\infty\le \sqrt{n} L_f M_f h^2.
\end{equation}
On the other hand, as in \cite[(B.6a), (B.6b)]{boba_etal}, thanks to (CA), for any $k$, there exists $\underline u_k$ such that
\begin{eqnarray}
\int_{t_k}^{t_{k+1}}f(y(t_k),\underline u(s))ds=h f(y(t_k),\underline u_k)\label{mean_f}\\
\int_{t_k}^{t_{k+1}}g(y(t_k),\underline u(s))e^{-\lambda s}ds\le h g(y(t_k),\underline u_k))e^{-\lambda t_k}\label{mean_g_i}.
\end{eqnarray}
From \eqref{mean_g_i} and \eqref{cota_g} we get
\begin{eqnarray}\label{mean_g}
\int_{t_k}^{t_{k+1}}(g(y(t_k),\underline u(s)-g(y(t_k),\underline u_k))e^{-\lambda s}ds&\le &
\int_{t_k}^{t_{k+1}}g(y(t_k),\underline u_k)(e^{-\lambda t_k}-e^{-\lambda s})ds\nonumber\\
&\le& \lambda M_g h^2 .
\end{eqnarray}
From \eqref{fal_1} and \eqref{mean_f}
$$
y(t)=y(t_k)+\int_{t_k}^t f(y(t_k),\underline u_k)ds+\epsilon_k,\quad \|\epsilon_k\|_\infty\le \sqrt{n}L_f M_f h^2.
$$
Let us denote by $y^{pc}$ the time-continuous trayectory solution with the same initial condition as $y$ associated
to the control $
\underline u^{pc}(t)=\underline u_k,\ \forall t\in[t_k,t_{k+1}).
$

Arguing as above we get
\begin{equation}\label{fal_2}
y^{pc}(t)=y^{pc}(t_k)+\int_{t_k}^t f(y^{pc}(t_k),\underline u_k)ds+\epsilon'_k,\quad \|\epsilon'_k\|_\infty\le \sqrt{n}L_f M_f h^2.
\end{equation}
Subtracting \eqref{fal_2} from \eqref{fal_1} and  using \eqref{lip_f} we obtain
\begin{eqnarray*}
\|y(t_k)-y^{pc}(t_k)\|_\infty&\le &\|y(t_{k-1})-y^{pc}(t_{k-1})\|_\infty+\sqrt{n}hL_f\|y(t_{k-1})-y^{pc}(t_{k-1})\|_\infty\\
&&\quad +\|\epsilon_{k-1}\|_\infty+\|\epsilon'_{k-1}\|_\infty\\
&\le&(1+ h\sqrt{n}L_f)\|y(t_{k-1})-y^{pc}(t_{k-1})\|_\infty+\|\epsilon_{k-1}\|_\infty+\|\epsilon'_{k-1}\|_\infty.
\end{eqnarray*}
Since $y(y_0)=y^{pc}(t_0)$ by standard recursion we get
\begin{equation}\label{fal_3}
\|y(t_k)-y^{pc}(t_k)\|_\infty\le e^{t_k\sqrt{n}L_f}\sum_{0\le l\le {k-1}}\left(\|\epsilon_l\|_\infty+\|\epsilon_l'\|_\infty\right)
\le 2 t_k e^{\sqrt{n}t_kL_f}\sqrt{n}L_f M_f h. 
\end{equation}
For the control $\underline u \in{\Bbb U}_{\rm ad}$ giving the minimum in \eqref{eq_v} and for 
$
\underline \bu=\left\{\underline u_0,\ldots \underline u_k,\ldots,\right\}.
$
we obtain
\begin{eqnarray*}
v_{h,k}(y)-v(y)\le \hat J_{h,k}(y,\underline \bu)-\hat J(y,\underline u)=
\hat J_{h,k}(y,\underline \bu)-\hat J(y^{pc},\underline u^{pc})+\hat J(y^{pc},\underline u^{pc})-\hat J(y,\underline u).
\end{eqnarray*}
The first term on the right-hand side above is bounded in Lemma \ref{le:con_hk_cons} so that
\begin{eqnarray*}
v_{h,k}(y)-v(y)\le C_1(h+k)+\hat J(y^{pc},\underline u^{pc})-\hat J(y,\underline u).
\end{eqnarray*}
To conclude we need to bound the second term.
We write
\begin{eqnarray}\label{fal_6}
\hat J(y^{pc},\underline u^{pc})-\hat J(y,\underline u)&=&
\int_0^T \left(g(y^{pc}(s),\underline u^{pc}(s))-g(y(s),\underline u(s))\right)e^{-\lambda s}ds\nonumber\\
&&\quad+\int_T^\infty \left(g(y^{pc}(s),\underline u^{pc}(s))-g(y(s),\underline u(s))\right)e^{-\lambda s}ds.
\end{eqnarray}
For the second term on the right-hand side above, applying \eqref{cota_g} we get
$$
\left|\int_T^\infty \left(g(y^{pc}(s),\underline u^{pc}(s))-g(y(s),\underline u(s))\right)e^{-\lambda s}ds\right|
\le 2 M_g \int_T^\infty e^{-\lambda s}ds=2M_g \frac{e^{-\lambda T}}{\lambda}.
$$
Let $$e^{-\lambda T}=h^{1/(1+\beta)},\quad \beta=\frac{\sqrt{n}L_f}{\lambda}.$$  Then
$$
T=\log(h^{-1/(1+\beta)\lambda}).
$$
We fix the above value of $T$ so that from \eqref{fal_6} we get
\begin{eqnarray}\label{fal_7}
\left|\hat J(y^{pc},\underline u^{pc})-\hat J(y,\underline u)\right|&\le&
\int_0^T \left(g(y^{pc}(s),\underline u^{pc}(s))-g(y(s),\underline u(s))\right)e^{-\lambda s}ds\nonumber\\
&&\quad+\frac{2M_g }{\lambda}h^{\frac{1}{1+\beta}} .\end{eqnarray}
To conclude we will bound the first term on the right-hand side of \eqref{fal_7}.

Also, for simplicity we assume there exists an integer $N$ such that $T=N\Delta t$.
For the first term on the right-hand side of \eqref{fal_7} we have
\begin{eqnarray*}
&&\int_0^T \left(g(y^{pc}(s),\underline u^{pc}(s))-g(y(s),\underline u(s))\right)e^{-\lambda s}ds=
\\
&&\quad \sum_{k=0}^{N-1}\int_{t_k}^{t_{k+1}}\left(g(y^{pc}(s),\underline u_k)-g(y(s),\underline u(s))\right)e^{-\lambda s}ds.
\end{eqnarray*}
Adding and subtracting terms we get
\begin{eqnarray*}
&&\int_{t_k}^{t_{k+1}}\left(g(y^{pc}(s),\underline u_k)-g(y(s),\underline u(s))\right)e^{-\lambda s}ds
=\\
&&\quad \int_{t_k}^{t_{k+1}}\left(g(y^{pc}(s),\underline u_k)-g(y^{pc}(t_k),\underline u_k)\right)e^{-\lambda s}ds
\\
&&\quad +\int_{t_k}^{t_{k+1}}\left(g(y^{pc}(t_k),\underline u_k)-g(y(t_k),\underline u_k)\right)e^{-\lambda s}ds\\
&&\quad +\int_{t_k}^{t_{k+1}}\left(g(y(t_k),\underline u_k)-g(y(t_k),\underline u(s))\right)e^{-\lambda s}ds
\\
 &&\quad +\int_{t_k}^{t_{k+1}}\left(g(y(t_k),\underline u(s))-g(y(s),\underline u(s))\right)e^{-\lambda s}ds
\end{eqnarray*}
Applying \eqref{lip_g}, \eqref{fal_3} and \eqref{mean_g} we get
\begin{eqnarray*}
&&\int_{t_k}^{t_{k+1}}\left(g(y^{pc}(s),\underline u_k)-g(y(s),\underline u(s))\right)e^{-\lambda s}ds
\le 2\sqrt{n}h^2 L_g M_f +\lambda h^2M_g\\
&&\quad+
2 n L_g h^2 t_k e^{\sqrt{n}t_kL_f}L_f M_f.
\end{eqnarray*}
And then
\begin{eqnarray*}
&&\int_0^T \left(g(y^{pc}(s),\underline u^{pc}(s))-g(y(s),\underline u(s))\right)e^{-\lambda s}ds
\le T h\left(2 \sqrt{n}L_g M_f +   \lambda M_g\right)\\
&&\quad+
2 n L_g T^2 h  e^{\sqrt{n} TL_f}L_f M_f.
\end{eqnarray*}
To conclude, we observe that with the definition of $T$ we get
$$
he^{\sqrt{n}T L_f}=h^{\frac{1}{1+\beta}},
$$
since we have chosen $T$ and $\beta$ to optimize the rate of convergence. The above term, together
with the last term in \eqref{fal_7} are the terms that produce a reduction in the rate of convergence compared with the finite horizon case.
We finally obtain
\begin{eqnarray*}
v_{h,k}(y)-v(y)\le C_1(h+k)+ C_2 \frac{1}{(1+\beta)^2\lambda^2}(\log(h))^2h^{\frac{1}{1+\beta}},\quad \beta=\frac{\sqrt{n}L_f}{\lambda}.
\end{eqnarray*}
\end{proof}
\begin{remark}
Let us observe that in view of \eqref{cota_buena_cons} and taking into account that $\beta$ is smaller than $1$ then we loose at most half an order in the rate of convergence in time
of the method up to a logarithmic term. This comes from adapting the arguments in \cite{boba_etal} in the
context of finite horizon problems to our infinite horizon case.
\end{remark}

  \end{document}


\maketitle
\abstract{We correct the definition of the characterization of the fully discrete functional defined in \cite[(18)-(19)]{nos}. As correctly proved in \cite{nos}, this fully discrete function is an approximation to the infinite horizon problem. However, the mistake is that the function defined in \cite[(20)]{nos} does not belong in general to the space $V^k$. In the present corrigenda, we introduce the right characterization of the fully discrete functional. We also correct the proofs in \cite{nos}. The only difference is an extra term in the error bound of order $k$ (which does not change the rate of convergence obtained in \cite{nos}) that comes from the difference between the functional introduced in \cite{nos} and the correct one.}
\bigskip

\section{Optimal bounds for the fully discrete approximations} The function defined in \cite[(20)]{nos} does not belong in general to the space $V^k$ since the controls $\bu$ are in general functions of $y$. To solve this problem, instead of \cite[(18)-(19)]{nos} we introduce now the correct definition of the fully discrete functional. 

For any $y\in {\Bbb R^n}$ and $\bu_j\in \cal U$, $1\le j\le n_s$, let us define the fully discrete functional
\begin{eqnarray}\label{eq:funcional_fd}
\hat J_{h,k}(y,\bu^1,\ldots,\bu^{n_s})&:=&h\sum_{n=0}^\infty \delta_h^n I_k g(\hat y_n,u_n^1,\ldots,u_n^{n_s}), \ \delta_h=(1-\lambda h),\\
\hat y_{n+1}&=&\hat y_n+h I_k f(\hat y_n,u_n^1,\ldots,u_n^{n_s}),\ \hat y_0=y,\label{eq:euler}
\end{eqnarray}
where using the barycentric coordinates
$$
y=\sum_{j=1}^{n_s}\mu_j(y) y_j, \ 0\le \mu_j\le 1,\ \sum_{j=1}^{n_s}\mu_j(y)=1,
$$
the interpolants $I_k g(y,u^1,\ldots,u^{n_s})$ and  $I_k f(y,u^1,\ldots,u^{n_s})$ are defined by
\begin{equation}\label{inter}
I_k g(y,u^1,\ldots,u^{n_s})=\sum_{j=1}^{n_s}\mu_j(y) g(y_j,u^j),\
I_k f(y,u^1,\ldots,u^{n_s})=\sum_{j=1}^{n_s}\mu_j(y) f(y_j,u^j).
\end{equation}
Now, we define $w_{h,k}\in V^k$ by
\begin{eqnarray}\label{eq:cha}
w_{h,k}(y)=\inf_{\bu^1,\ldots, \bu^{n_s}}  \hat J_{h,k}(y,\bu^1,\ldots,\bu^{n_s}).
\end{eqnarray}
The statement of \cite[Theorem 3]{nos} has to be changed by the following while the proof is essentially the same.
\begin{Theorem}
For any $y\in \overline \Omega$ the function $w_{h,k}(y)$ satifies the equation

$$w_{h,k}(y)=\inf_{u^1,\dots,u^{n_s}}\{hI_kg(y,u^1,\dots,u^{n_s})+\delta_hw_{h,k}(y+hI_kf(y,u^1,\dots,u^{n_s}))\}$$
and, as a consequence, for each node $y_i$, $i=1,\dots,n_s$

$$w_{h,k}(y_i)=\inf_{u^i}\{hg(y_i,u^i)+\delta_h w_{h,k}(y_i+hf(y_i,u^i))\}.$$
\end{Theorem}
The following new theorem can be easily proved. 
\begin{Theorem}\label{th:acla}
For each node $y_i$, $i=1,\dots,n_s,$
let us denote by $u^i$ the argument giving the minimum in
$$w_{h,k}(y_i)=\inf_{u^i}\{hg(y_i,u^i)+\delta_h w_{h,k}(y_i+hf(y_i,u^i))\}.$$
Then, for any $y\in \overline \Omega$ the sequences $\bu^i$, $i=1,\ldots,n_s$ giving the minimum in
$$w_{h,k}(y)=\inf_{\bu^1,\ldots, \bu^{n_s}}  \hat J_{h,k}(y,\bu^1,\ldots,\bu^{n_s})$$
are determined by the values of the controls at the nodes, $u^i$, $1\le i\le n_s$. More precisely, for any $n$, the values $\left\{u_n^1,\ldots,u_n^{n_s}\right\}$  in \eqref{eq:funcional_fd}-\eqref{eq:euler} are $u_n^i=u^i$ if $\mu_i(\hat y^n)\neq 0$ and $u_n^i=0$ if $\mu_i(\hat y^n)=0.$ 
\end{Theorem}
From the above theorem is clear that the interpolants defined in \eqref{inter} are always piecewise functions in $V^k$. Let us observe that $I_k g(y,u^1,\ldots,u^{n_s})$ and $I_k f(y,u^1,\ldots,u^{n_s})$ are
just the linear interpolants of functions $g(y,u(y))$ and $f(y,u(y))$, respectively. Then,
it is immediate to prove $w_{h,k}\in V^k$ which implies $w_{h,k}$ is the unique solution defined by \cite[(17)]{nos}, i.e., $w_{h,k}=v_{h,k}$ and, as a consequence, gives the right characterization of the fully discrete functional.

The statement of \cite[Lemma 1]{nos} should be now the following while the proof is exactly the same.
\begin{lema}\label{le:con}
Let $\hat J$ and $\hat J_{h,k}$ be the functionals defined in (10) and \eqref{eq:funcional_fd} respectively.
Assume conditions  (3), (4), (5), (7), (8), (9) and (23) hold. Let $y_0=y\in \overline \Omega$. Then
\begin{eqnarray}\label{eq_lim0}
\lim_{h\rightarrow 0, k\rightarrow 0} |\hat J(y_0,u)-\hat J_{h,k}(y_0,\bu^1,\ldots,\bu^{n_s})|=0,
\end{eqnarray}
where $u\in {\Bbb U}_{\rm ad}$ and for $i=1,\ldots,n_s$, $\bu^i=\bu=\left\{u_0,u_1,\ldots,\right\}=\left\{u(t_0),u(t_1),\ldots,\right\}$, $t_i=ih$, $i=0,1,\ldots,$.
\end{lema}
The same correction applies to Lemmas 2, 3 and 4. 

We give the correction of the proof of \cite[Theorem 4]{nos}.

{\bf Proof of Theorem 4.}
For any $y\in  \overline \Omega$, in view of \eqref{eq:cha}
let us denote by $\left\{\overline \bu^1,\ldots,\overline \bu^{n_s}\right\}$, $\overline\bu\in {\cal U}$ a control giving the minimum
$$
v_{h,k}(y)=\hat J_{h,k}(y,\overline \bu^1,\ldots,\overline\bu^{n_s}),
$$
with $\left\{\overline u^1,\ldots,\overline u^{n_s}\right\}$ the controls giving the minimum at the nodes.

Let us define $\overline \bu=\left\{\overline u_0,\ldots\overline u_n\,\ldots\right\}$ as follows. For $n=0,1,\ldots$, 
and  $\hat y_n$ defined in \eqref{eq:euler} written as 
$\hat y_n=\sum_{j=1}^{n_s}{\mu_j(\hat y_n)}y_j$
then
$\overline u_n=\sum_{j=1}^{n_s}{\mu_j(\hat y_n)} \overline u^j$.
Let us denote by
$$
\hat J_{h,k}(y,\overline \bu)=\hat J_{h,k}(y,\overline \bu,\ldots,\overline \bu).
$$
Then
\begin{equation}\label{cotamas}
v(y)-v_{h,k}(y)\le \left( J(y,\overline u)-\hat J_{h,k}(y,\overline \bu)\right)
+\left(\hat J_{h,k}(y,\overline \bu))-\hat J_{h,k}(y,\overline \bu^1,\ldots,\overline\bu^{n_s})\right),
\end{equation}
where $\overline u\in{\Bbb U}_{\rm ad}$ such that $\overline u(t_j)=\overline u_j$.
Applying \cite[(42)]{nos} to the first term, there exist  positive constants ${C_1=
C(\lambda,M_f,M_g,L_f,L_g)}$ and ${C_2=
C(\lambda,L_f,L_g)}$ such that
\begin{equation}\label{puessi}
v(y)-v_{h,k}(y)\le C_1(h+k)+{C_2M_u}+\left(\hat J_{h,k}(y,\overline \bu))-\hat J_{h,k}(y,\overline \bu^1,\ldots,\overline\bu^{n_s})\right).
\end{equation}
The second term on the right-hand side of \eqref{puessi} can be bounded by
$C(\lambda,M_f,Mg,L_f,L_g)k$ arguing as in the proof of \cite[Lemma 2]{nos} and applying standard interpolation arguments. The
bound depends on $L_{\overline u}$ such that
$$
\|u^i-u^j\|_2\le L_{\overline u}\|y_i-y_j\|_2,\quad i,j=1,\ldots,n_s.
$$
This is the extra term that appears in the error analysis, as mentioned in the abstract.

The rest of the proof remains unchanged defining
\begin{eqnarray}\label{def}
\hat J_{h,k}(y,\underline \bu)=\hat J_{h,k}(y,\underline \bu^1,\ldots,\underline \bu^{n_s}),
\end{eqnarray}
for $\underline \bu^j=\underline \bu$, $j=1,\ldots,n_s$.

Finally, we give the correction of the proof of \cite[Theorem 7]{nos}.

{\bf Proof of Theorem 7.} The same correction as in the proof of \cite[Theorem 4]{nos} applies to this theorem. In view of \eqref{eq:cha}
let us denote by $\left\{\overline \bu^1,\ldots,\overline \bu^{n_s}\right\}$, $\overline\bu^i\in {\cal U}$
a control giving the minimum
$$
v_{h,k}(y)=\hat J_{h,k}(y,\overline \bu^1,\ldots\overline \bu^{n_s}),
$$
with $\left\{\overline u^1,\ldots,\overline u^{n_s}\right\}$ the controls giving the minimum at the nodes.

Let us define $\overline \bu=\left\{\overline u_0,\ldots\overline u_n,\ldots\,\right\}$ as follows. For $n=0,1,\ldots$, 
and  $\hat y_n$ defined in \eqref{eq:euler} written as 
$\hat y_n=\sum_{j=1}^{n_s}{\mu_j(\hat y_n)}y^j$
then
$\overline u_n=\sum_{j=1}^{n_s}{\mu_j(\hat y_n)} \overline u^j$.
Using \eqref{def} to define
$
\hat J_{h,k}(y,\overline \bu)=\hat J_{h,k}(y,\overline \bu,\ldots,\overline \bu).
$
Then
$$
v(y)-v_{h,k}(y)\le \left(\hat J(y,\overline u)-\hat J_{h,k}(y,\overline \bu)\right)+\left(\hat J_{h,k}(y,\overline \bu)-\hat J_{h,k}(y,\overline \bu^1,\ldots\overline \bu^{n_s})\right),
$$
where $\overline u\in{\Bbb U}_{\rm ad}^{pc}$ such that $\overline u(t)=\overline u_i$, $t\in[t_i,t_{i+1})$.

To conclude we apply \cite[(50)]{nos} to the first term on the right-hand side above and bound the second one with the same argument we applied to the second term in \eqref{puessi}.
The rest of the proof remains unchanged using definition \eqref{def}.